\documentclass[a4paper,12pt]{article}

\usepackage[textwidth=125mm, textheight=195mm]{geometry}
\usepackage[english]{babel}
\usepackage{graphicx}
\usepackage{amsmath}
\usepackage{amsfonts}
\usepackage{amsthm}
\usepackage{epstopdf}
\usepackage{color}

\begin{document}

\newcommand{\wk}{\mbox{$\,<$\hspace{-5pt}\footnotesize )$\,$}}

\numberwithin{equation}{section}
\newtheorem{teo}{Theorem}
\newtheorem{lemma}{Lemma}

\newtheorem{coro}{Corollary}
\newtheorem{prop}{Proposition}
\theoremstyle{definition}
\newtheorem{definition}{Definition}
\theoremstyle{remark}
\newtheorem{remark}{Remark}

\newtheorem{scho}{Scholium}
\newtheorem{open}{Question}
\newtheorem{example}{Example}
\numberwithin{example}{section}
\numberwithin{lemma}{section}
\numberwithin{prop}{section}
\numberwithin{teo}{section}
\numberwithin{definition}{section}
\numberwithin{coro}{section}
\numberwithin{figure}{section}
\numberwithin{remark}{section}
\numberwithin{scho}{section}

\bibliographystyle{abbrv}

\title{Surface immersions in normed spaces from the affine point of view}
\date{}

\author{Vitor Balestro\footnote{Corresponding author}  \\ CEFET/RJ Campus Nova Friburgo \\ 28635000 Nova Friburgo \\ Brazil \\ vitorbalestro@id.uff.br \and Horst Martini \\ Fakult\"{a}t f\"{u}r Mathematik \\ Technische Universit\"{a}t Chemnitz \\ 09107 Chemnitz\\ Germany \\ martini@mathematik.tu-chemnitz.de \and  Ralph Teixeira \\ Instituto de Matem\'{a}tica e Estat\'{i}stica  \\ Universidade Federal Fluminense \\24210201 Niter\'{o}i\\ Brazil \\ ralph@mat.uff.br}

\maketitle

\begin{abstract} The aim of this paper is to investigate the differential geometry of immersed surfaces in three-dimensional normed spaces from the viewpoint of affine differential geometry. We endow the surface with a useful Riemannian metric which is closely related to normal curvature, and from this we re-calculate the Minkowski Gaussian and mean curvatures. These curvatures are also re-obtained in terms of ambient affine distance functions, and as a consequence we characterize minimal surfaces as the solutions of a certain differential equation. We also investigate in which cases it is possible that the affine normal and the Birkhoff normal vector fields of an immersion coincide, proving that this only happens when the geometry is Euclidean.

\end{abstract}


\bigskip


\section{Introduction}

The differential geometry of normed spaces is a topic of research that was studied by authors like Busemann \cite{Bus3}, Guggenheimer \cite{Gug2}, and Petty \cite{Pet}, and it is still far away from being comprehensively investigated. Its relation to Finsler geometry is nicely described in \cite{Bus2} and \cite{Bus7}; see also the more recent references \cite{shen} and \cite{shen2}. This paper is the second of a series of three papers devoted to study this topic (the other two papers are \cite{diffgeom} and \cite{diffgeom3}, see also \cite{Ba-Ma-Sho}). In the first paper \cite{diffgeom} we studied the differential geometry of surfaces immersed in normed spaces from the viewpoint of classical differential geometry. However, the methods used to define some curvature concepts came from affine differential geometry, and hence many questions related to this latter subject emerged. In this present paper we aim to address and answer some of these questions. \\

We begin by briefly describing the theory developed in \cite{diffgeom}. We work with an immersion $f:M\rightarrow(\mathbb{R}^3,||\cdot||)$ of a surface $M$ in the space $\mathbb{R}^3$ endowed with a norm $||\cdot||$, which is considered to be \emph{admissible}. This means that the \emph{unit sphere} $\partial B:=\{x\in\mathbb{R}^3:||x||=1\}$ of the \emph{normed} or \emph{Minkowski space} $(\mathbb{R}^3,||\cdot||)$ has strictly positive Gaussian curvature as a surface of the Euclidean space $(\mathbb{R}^3,\langle \cdot,\cdot\rangle)$, where $\langle\cdot,\cdot\rangle$ denotes the usual \emph{inner product} in $\mathbb{R}^3$. Note that the unit sphere is the boundary of the \emph{unit ball} $B:=\{x \in \mathbb{R}^3:||x|| \leq 1\}$, which is a compact, convex set with interior points centered at the origin. Respective homothetical copies are called \emph{Minkowski spheres} and \emph{Minkowski balls}. We say that a vector $v\in\mathbb{R}^3$ is \emph{Birkhoff orthogonal} to a plane $P \subseteq \mathbb{R}^3$ if for each $w \in P$ we have $||v + tw|| \geq ||v||$ for any $t \in \mathbb{R}$ (see \cite{alonso}). Geometrically, a vector $v$ is Birkhoff orthogonal to a plane $P$ if $P$ supports the unit ball of $(\mathbb{R}^3,||\cdot||)$ at $v/||v||$. Due to the admissibility of the norm, it follows that Birkhoff orthogonality is unique both on left and on right. \\

The \emph{Birkhoff-Gauss map} of $M$ is an analogue to the Gauss map defined in terms of Birkhoff orthogonality as follows: for each $p \in M$, the Birkhoff normal vector to $M$ at $p$ is a vector $\eta(p) \in \partial B$ which is Birkhoff orthogonal to the tangent plane to $M$ at $p$. Such a vector field can be globally defined if $M$ is orientable, and hence we will always assume this hypothesis. The immersion $f:M\rightarrow(\mathbb{R}^3,||\cdot||)$ with the Birkhoff normal vector field is an \emph{equiaffine immersion}, in the sense of \cite{nomizu} (see \cite{diffgeom} for a proof). \\

At each point, the eigenvalues of the differential map $d\eta_p$ are called \emph{principal curvatures}. Their product is the \emph{Minkowski Gaussian curvature}, and their arithmetic mean is the \emph{Minkowski mean curvature}. We also endow $M$ with an induced connection $\nabla$ by means of the \emph{Gauss equation}
\begin{align*} D_XY = \nabla_XY + h(X,Y)\eta,
\end{align*}
where $X,Y$ are smooth vector fields in $M$, and $h(X,Y)$ is a symmetric bilinear form which can be regarded as the second fundamental form in our context. We say that the immersion is \emph{nondegenerate} if the rank of $h$ equals $2$.For this bilinear form, we have the formula
\begin{align}\label{exph} h(X,Y) = \frac{\langle D_XY,\xi\rangle}{\langle \eta,\xi\rangle} = -\frac{\langle Y,d\xi_pX\rangle}{\langle\eta,\xi\rangle} = -\frac{\langle du^{-1}_{\eta(p)}Y,d\eta_pX\rangle}{\langle\eta,\xi\rangle},
\end{align}
where $\xi$ denotes the usual Euclidean Gauss map of $M$, and $u^{-1}$ is the Euclidean Gauss map of the unit sphere $\partial B$. Notice that we have $\eta = u\circ\xi$ (where $\circ$ denotes the usual composition of maps). We also define the \emph{normal curvature} $k_{M,p}(X)$ of $M$ at a point $p$ in direction $X$ to be the circular curvature of the curve obtained by intersecting $M$ with the plane spanned by $\eta(p)$ and $X$ (translated to pass through $p$, of course). For the normal curvature we have the equality
\begin{align}\label{normcurv} k_{M,p}(X) = \frac{\langle du^{-1}_{\eta(p)}X,d\eta_pX\rangle}{\langle du^{-1}_{\eta(p)}X,X\rangle}.
\end{align}

Now we describe the structure of the paper. In Section \ref{riemann} we endow the surface with a Riemannian metric which has a lot of interesting relations with its Minkowski normal curvature. Such metric will be very useful in Section \ref{distance}, where we re-obtain the (Minkowski) curvature concepts by means of ambient affine distance functions. Section \ref{blaschke} is devoted to the question when the Birkhoff normal field of a surface coincides with the \emph{affine normal field}. We prove that this is the case if and only if the ambient geometry is Euclidean and the surface is a Euclidean sphere. \\

As mentioned, in this paper we will continue our considerations in \cite{diffgeom}, dealing further on with these concepts. Other important references devoted to \emph{Minkowski geometry} (i.e., the geometry of finite dimensional real Banach spaces) are \cite{martini2}, \cite{martini1}, and \cite{thompson}. Regarding differential geometry, our main references are \cite{manfredo} and \cite{nomizu}. 

\section{A related Riemannian metric}\label{riemann}

In this section we endow an immersed surface with a certain Riemannian metric which appears naturally when studying the Minkowski normal curvature of a surface. \\

Let $f:M\rightarrow\mathbb{R}^3$ be an immersed surface with (Euclidean) Gauss map given by $\xi:M\rightarrow\partial B_e$. The \emph{Dupin indicatrix} of $M$ at $p$ is the curve in $T_pM$ formed by the vectors $V \in T_pM$ such that $\langle V,d\xi_pV\rangle = \pm1$. Since the unit sphere of the Minkowski norm is an immersed surface whose Gauss map is $u^{-1}$, we have that for each $q \in \partial B$ its Dupin indicatrix is determined by the solution of the equation $\langle du^{-1}_qV,V\rangle = 1$ in $T_q\partial B$ (where we may consider only the positive sign, since we are assuming that the norm is admissible, and hence the Gaussian curvature of $\partial B$ is strictly positive). It follows that the Dupin indicatrix of $\partial B$ at each point is an ellipse, and therefore induces a Euclidean metric (which may differ, however, from the ambient Euclidean metric). We will endow an immersed surface with a Riemannian metric by considering, in each of its tangent spaces, the metric given by the Dupin indicatrix of $\partial B$ at the parallel tangent space. At first glance it seems that this is a somewhat artificial construction, but we can sharply describe the Dupin indicatrix of the \textbf{Minkowski sphere} in terms of the principal directions of the \textbf{surface}.\\

Let $f:M\rightarrow(\mathbb{R}^3,||\cdot||)$ be an immersed surface, let $p \in M$, and let $V_1,V_2 \in T_pM$ be principal directions associated to the (Minkowski) principal curvatures $\lambda_1,\lambda_2 \in \mathbb{R}$, respectively. We may assume that $h(V_1,V_2) = 0$, since this is the case when $p$ is non-umbilic; if $p$ is umbilic, then every direction is principal (see \cite[Section 4]{diffgeom}). Now we re-scale $V_1$ and $V_2$ in order to have $\langle du^{-1}_{\eta(p)}V_1,V_1\rangle = \langle du^{-1}_{\eta(p)}V_2,V_2\rangle = 1$, where we remember the identification $T_{\eta(p)}\partial B \simeq T_pM$. From the proof of \cite[Theorem 5.2]{diffgeom} we have that $\langle du^{-1}_{\eta(p)}V_1,V_2 \rangle = \langle du^{-1}_{\eta(p)}V_2,V_1\rangle = 0$. Hence the Dupin indicatrix of $\partial B$ at $\eta(p)$ is the curve parametrized as
\begin{align}\label{dupin} [0,2\pi]\ni\theta \mapsto V(\theta) = V_1\cos\theta + V_2\sin\theta \in T_{\eta(p)}\partial B\simeq T_pM.
\end{align}
As mentioned previously, it is clear that this curve is an ellipse, and it is the unit circle of the metric induced by the inner product $\langle\cdot,\cdot\rangle_p:T_pM\times T_pM\rightarrow\mathbb{R}$ defined by the setting $\langle V_1,V_1\rangle_p = \langle V_2,V_2\rangle_p = 1$ and $\langle V_1,V_2\rangle_p = 0$ (which is merely the inner product $(X,Y)\mapsto\langle du^{-1}_{\eta(p)}X,Y\rangle$). From now on, we refer to this curve as the \emph{Dupin indicatrix at} $T_pM$, and to the associated metric as the \emph{Dupin metric at} $T_pM$. Notice that, in the classical setting, by this construction one would naturally re-obtain the restriction of the ambient metric to each tangent space. \\

The first application of this metric is a way to calculate the Minkowski mean curvature, analogous to the Euclidean subcase. In this particular case, it is known that the mean curvature can be calculated as the mean of the normal curvature over the unit circle of the tangent space. In other words, if $k_n(\theta)$ denotes the (Euclidean) normal curvature in the direction of a vector forming an angle $\theta$ with a fixed direction, then
\begin{align*} H_e = \frac{1}{2\pi}\int_0^{2\pi}k_n(\theta) \ d\theta,
\end{align*}
where $H_e$ denotes the Euclidean mean curvature of $M$. We obtain something similar for the general Minkowksi case, but now with the Dupin indicatrix of $T_pM$. In what follows, $H$ denotes the Minkowski mean curvature of $M$. 

\begin{prop} The Minkowski mean curvature of an immersed surface $M$ at a point $p \in M$ is the mean of the normal curvature as a function of the Dupin indicatrix of $T_pM$. 
\end{prop}
\begin{proof} From (\ref{normcurv}) we get the equality 
\begin{align*} k_{M,p}(\theta) := k_{M,p}(V(\theta)) = \lambda_1(\cos\theta)^2 + \lambda_2(\sin\theta)^2.
\end{align*}
Hence, a simple calculation gives
\begin{align*} \frac{1}{2\pi}\int_{0}^{2\pi}k_{M,p}(\theta) \ d\theta = \frac{\lambda_1}{2\pi}\int_{0}^{2\pi}(\cos\theta)^2d\theta + \frac{\lambda_2}{2\pi}\int_0^{2\pi}(\sin\theta)^2d\theta = \frac{\lambda_1+\lambda_2}{2} = H,
\end{align*}
as we claimed. Notice that $2\pi$ equals the length of the Dupin indicatrix in the metric derived from it. 

\end{proof}

The Dupin metric in a tangent space $T_pM$ gives rise to a natural orthogonality relation: we say that $X,Y \in T_pM$ are \emph{Dupin orthogonal} whenever $\langle X,Y\rangle_p = 0$. In the Euclidean subcase, it is well known that the sum of the normal curvatures of $M$ at $p$ in a pair of orthogonal directions equals twice the mean curvature. We will show that this is true in the Minkowski case if one replaces usual orthogonality by Dupin orthogonality.

\begin{prop} Let $X,Y \in T_pM$ be non-zero vectors. Then
\begin{align*} k_{M,p}(X) + k_{M,p}(Y) = 2H
\end{align*}
if and only if $X$ and $Y$ are Dupin orthogonal or Dupin complementary.
\end{prop}
\begin{proof} Let $X$ and $Y$ be given in the Dupin indicatrix of $T_pM$ as $V(\theta_0)$ and $V(\theta_1)$, respectively. Then Dupin orthogonality and Dupin complementarity of $X$ and $Y$ give $\cos(\theta_0 - \theta_1) = 0$ and $\cos(\theta_0+\theta_1) = 0$, respectively. Hence we may write $\theta_1 = \theta_0 \pm \frac{\pi}{2}$. It follows that
\begin{align*} k_{M,p}(X) + k_{M,p}(Y) = \lambda_1(\cos^2\theta_0+ \sin^2\theta_0)+\lambda_2(\cos^2\theta_0+ \sin^2\theta_0) = \lambda_1+\lambda_2 = 2H.
\end{align*}
Now assume that $k_{M,p}(X) + k_{M,p}(Y) = 2H$. Suppose also that $\lambda_2\neq\lambda_1$, since the equality case is trivial. Then we may write
\begin{align*} \lambda_1(\cos^2\theta_0 + \cos^2\theta_1) + \lambda_2(\sin^2\theta_0 + \sin^2\theta_1) = \lambda_1 + \lambda_2,
\end{align*}
and this can be easily rewritten as
\begin{align*} (\lambda_2-\lambda_1)(\sin^2\theta_0 + \sin^2\theta_1) = \lambda_2-\lambda_1.
\end{align*}
It follows that $\cos(\theta_0-\theta_1) = 0$ or $\cos(\theta_0+\theta_1) = 0$, and therefore $X$ and $Y$ are Dupin orthogonal or Dupin complementary.

\end{proof}

\begin{coro} Let $f:M\rightarrow(\mathbb{R}^3,||\cdot||)$ be an immersed hypersurface whose Minkowski Gaussian curvature is negative. If the Minkowski mean curvature of $M$ at $p \in M$ equals $0$, then the asymptotic directions of $M$ at $p$ are Dupin orthogonal.
\end{coro}
\begin{proof} One just has to recall that, due to \cite[Corollary 5.3]{diffgeom}, a direction $X \in T_pM$ is \emph{asymptotic} if and only if $k_{M,p}(X) = 0$. Hence the result comes directly from the proposition above. 

\end{proof}

\begin{remark} In the Euclidean subcase, the sum of normal curvatures of an immersed surface $M$ at a point $p \in M$ at two orthogonal directions is always a constant (see \cite{manfredo}). In the general Minkowski case, the Dupin indicatrix at each tangent plane (which depends only on the ambient Minkowski metric, and not on the surface) somehow ``organizes" the directions in an analogous way. \\
\end{remark}

From the affine viewpoint, it will be better to work with another Riemannian metric, obtained from the Dupin metric (and which preserves its orthogonality relation). We call this new metric the \emph{weighted Dupin metric}, and it is defined as
\begin{align*} b(X,Y) := \frac{\langle du^{-1}_{\eta(p)}X,Y\rangle}{\langle\eta(p),\xi(p)\rangle},
\end{align*} 
for $p \in M$ and $X,Y \in T_pM$. Now we will establish a relation between the weighted Dupin metric of an immersed surface and its Minkowski Gaussian curvature. In classical differential geometry, it is fairly known that the (Euclidean) Gaussian curvature of a surface equals the ratio between the determinants (with respect to any fixed basis) of the second and first fundamental forms, respectively. 

\begin{prop} Let $h_{ij}$ and $b_{ij}$ denote the affine fundamental form $h$ and the weighted Dupin metric in a certain fixed basis. Then 
\begin{align*} K = \frac{\mathrm{det}(h_{ij})}{\mathrm{det}(b_{ij})},
\end{align*}
where $K$ is the Minkowski Gaussian curvature.
\end{prop}
\begin{proof} We can assume that the Gaussian curvature is non-zero, since this case is straightforward. By assuming this, we can take a local frame $\{V_1,V_2\}$ of (Minkowski) principal directions, i.e., such that $d\eta_qV_1 = \lambda_1V_1$ and $d\eta_qV_2 = \lambda_2V_2$, for each $q$ in a small neighborhood where the principal curvatures $\lambda_1$ and $\lambda_2$ do not vanish. From the proof of \cite[Theorem 5.2]{diffgeom} we have
\begin{align*} \langle du^{-1}_{\eta(p)}V_1,V_2\rangle = \langle V_1,du^{-1}_{\eta(p)}V_2\rangle = 0,
\end{align*}
and since $V_1$ and $V_2$ are conjugate in the classical sense (see \cite[Lemma 4.2]{diffgeom}), we also have $h(V_1,V_2) = 0$ (in case that $p$ is umbilic, we just choose $V_1$ and $V_2$ to be conjugate). Now we compute
\begin{align*} h(V_1,V_1) = -\frac{\langle V_1,d\xi_pV_1\rangle}{\langle \eta,\xi\rangle} = -\frac{\langle V_1,du^{-1}_{\eta(p)}\circ d\eta_pV_1\rangle}{\langle \eta,\xi\rangle} = -\frac{\lambda_1\langle V_1,du^{-1}_{\eta(p)}V_1\rangle}{\langle\eta,\xi\rangle},
\end{align*}
and an analogue holds for $h(V_2,V_2)$. Finally, in the basis $\{V_1,V_2\}$ we have
\begin{align*} \frac{\mathrm{det}(h_{ij})}{\mathrm{det}(b_{ij})} = \frac{h(V_1,V_1)h(V_2,V_2)}{b(V_1,V_1)b(V_2,V_2)} = \frac{\lambda_1\lambda_2\langle du^{-1}_{\eta(p)}V_1,V_1\rangle \langle du^{-1}_{\eta(p)}V_2,V_2\rangle \langle\eta,\xi\rangle^{-2}}{\langle du^{-1}_{\eta(p)}V_1,V_1\rangle \langle du^{-1}_{\eta(p)}V_2,V_2\rangle \langle\eta,\xi\rangle^{-2}} =  \lambda_1\lambda_2,
\end{align*}
and the latter is the Minkowski Gaussian curvature. This is what we wanted to verify. 

\end{proof}

\section{Distance functions and curvatures}\label{distance}

In this section we obtain the Minkowski curvatures of a surface in terms of distance functions. Recall that a point $p \in M$ of the domain of a function $g:M\rightarrow\mathbb{R}$ is said to be a \emph{critical point} if $dg_p = 0$. Let $p \in M$ be a fixed point. For each $q \in M$, we can decompose the vector $p-q$ as follows:
\begin{align}\label{distfunc} p-q = g(q)\eta(p) + V(q),
\end{align}
where $V(q) \in T_pM$ is the projection of $q-p$ on $T_pM$ and $g:M\rightarrow\mathbb{R}$ is a smooth function. The function $g$ can be regarded as the Minkowski distance from $q \in M$ to the plane $p\oplus T_pM$, where $\oplus$ denotes the direct sum (geometrically, we are simply translating $T_pM$ such that its origin lies at $p$).
\begin{lemma} The point $p \in M$ is a critical point of the function $g$ defined above. 
\end{lemma}
\begin{proof} Let $X$ denote a smooth extension of a fixed vector $X \in T_pM$ in a neighborhood of $p$ (with a little abuse of notation). Differentiating (\ref{distfunc}) with respect to $X$, and evaluating at $p$, we have
\begin{align*} -X = (Xg)\eta + D_XV = (Xg)\eta + \nabla_XV + h(X,V)\eta.
\end{align*}
It follows that $dg_pX = -h(X,V) = 0$, since $V(p) = 0$. 

\end{proof}

In standard affine differential geometry, one can define an analogue of the classical Hessian in a manifold $M$ endowed with a nondegenerate bilinear form $h:M\times M\rightarrow\mathbb{R}$. Namely, the \emph{h-Hessian} $\mathrm{hess}_hf$ of a map $f:(M,h)\rightarrow\mathbb{R}$ is defined as 
\begin{align}\label{hessh}
\mathrm{hess}_hf(X,Y) := X(Yf) - (\bar{\nabla}_XY)f, 
\end{align}
for any $p \in M$ and any $X,Y \in T_pM$, where $\bar{\nabla}$ is the Levi-Civita connection of $h$. For our purpose, we will consider the Hessian in $M$ with respect to the weighted Dupin metric. Denoting by $\hat{\nabla}$ the Levi-Civita connection of the weighted Dupin metric, we define the \emph{b-Hessian} of the immersion $f:M\rightarrow\mathbb{R}$ to be
\begin{align}\label{hessb} \mathrm{hess}_bf(X,Y) := X(Yf) - (\hat{\nabla}_XY)f.
\end{align}

As a consequence of the previous lemma, it follows that the $b$-Hessian $\mathrm{hess}_bg$ of $g$ at $p$ is given by $\mathrm{hess}_bg(X,Y)|_p = X(Yg)|_p$, since $p$ is a critical point of $g$. 

\begin{teo} The Hessian of the function $g$ defined above at $p \in M$ equals $-h$ at $p$, where $h$ is the affine fundamental form. In particular, if $X$ is unit in the weighted Dupin metric, then
\begin{align*}  k_{M,p}(X) = \mathrm{hess}_bg(X,X)|_p.
\end{align*}
\end{teo}
\begin{proof} We evaluate the (Euclidean) inner product of (\ref{distfunc}) with $\xi(p)$ to obtain
\begin{align*} \langle p-q,\xi(p)\rangle = g(q)\langle \eta(p),\xi(p)\rangle.
\end{align*}
Let $X,Y \in T_pM$ and denote by the same letters smooth extensions of these vectors to a neighborhood of $p$. Derivating the above expression with respect to $Y$ and $X$, respectively, and evaluating at $p$ yields
\begin{align*} -\langle D_XY,\xi\rangle = X(Y(g))\langle\eta,\xi\rangle,
\end{align*}
where the reader may notice that $\eta$ and $\xi$ are always evaluated at $p$; that is why their derivatives vanish. Since $p$ is a critical point of $g$, from (\ref{exph}) we get
\begin{align}\label{hessandh} \mathrm{hess}_bg(X,Y)|_p = X(Yg)|_p = - \frac{\langle D_XY,\xi\rangle}{\langle \eta,\xi\rangle} = -h(X,Y).
\end{align}
The claim on the normal curvature comes straightforwardly from the formula
\begin{align}\label{normh} k_{M,p}(X) = -\frac{h(X,X)\langle\eta,\xi\rangle}{\langle du^{-1}_{\eta(p)}X,X\rangle},
\end{align}
obtained in \cite[Corollary 5.3]{diffgeom}. The reader may notice that here we are generalizing a well known result from classical differential geometry by regarding the affine fundamental form as the second fundamental form, and normalizing with respect to the weighted Dupin metric (instead of the usual metric). 

\end{proof}

\begin{remark} Notice that equality (\ref{normh}) can be written as 
\begin{align*} k_{M,p}(X) = -\frac{h(X,X)}{b(X,X)},
\end{align*}
which makes the Minkowski normal curvature analogous to the usual Euclidean normal curvature if one regards $h$ as the second fundamental form and the weighted Dupin metric as the first fundamental form. Of course, this is indeed the case if the norm in $\mathbb{R}^3$ is Euclidean.\\
\end{remark}

As in the Euclidean subcase, we will obtain the curvatures of an immersed surface in terms of the distances of the points of the surface to a fixed point in $\mathbb{R}^3$. Let $a \in \mathbb{R}^3\setminus M$ be a fixed point, and define the \emph{distance function}  $D_a:M\rightarrow\mathbb{R}$  of $M$ to $a$ as $D_a(q) = ||q-a||$, for $q \in M$. Notice that the level sets of $D_a$ are the spheres $S_{\rho}(a):=\{x\in \mathbb{R}^3:||x-a|| = \rho\}$, $\rho \geq 0$, and hence a point $p \in M$ is a critical point of $D_a$ if and only if $T_pS_{||p-a||}(a) = T_pM$. In other words, $p \in M$ is a critical point of $D_a$ if and only if $p-a$ is Birkhoff orthogonal to $T_pM$. \\

\begin{prop} Let $p\in M$ be a critical point of the distance function $D_a:M\rightarrow\mathbb{R}$, where $a \in \mathbb{R}^3\setminus M$. Then we have the equivalence
\begin{align*} \mathrm{hess}_bD_a(V,V)|_p = 0 \ \Leftrightarrow \ k_{M,p}(V) = \frac{1}{D_a(p)}
\end{align*}
for any nonzero vector $V \in T_pM$. 
\end{prop}
\begin{proof} Let $\gamma:(-\varepsilon,\varepsilon)\rightarrow M$ be an arc-length parametrization of the curve obtained as intersection of $M$ with the (translated, to pass through $p$) plane spanned by $\eta(p)$ and $V$. Assume also that $\gamma(0) = p$ and that $\gamma'(0) = V$ (in other words, we are assuming that $V$ is unit). By definition, we have that the normal curvature $k_{M,p}(V)$ is the circular curvature of $\gamma(s)$ at $s = 0$ (see \cite{diffgeom}). Observe that, since $p$ is a critical point of $D_a$, it follows that $p - a$ is in the direction of $\eta(p)$, and hence $a$ is a point of the (translated, to pass through $p$) plane spanned by $V$ and $\eta(p)$. From \cite[Proposition 9.1]{Ba-Ma-Sho} we have immediately that
\begin{align*} \left.\frac{d^2}{ds^2}\left(D_a\circ\gamma\right)\right|_{s=0} = 0 \  \Leftrightarrow \  k_{M,p}(V) = \frac{1}{D_a(p)}.
\end{align*}
Since we clearly have that $\mathrm{hess}_bD_a(V,V)|_p$ is precisely the expression on the left hand side of the equality above, the proof is complete. 

\end{proof}

\begin{remark} A concept of \emph{affine normal curvature} for \emph{Blaschke immersions} is defined in \cite[Definition 5.2]{davis}. The reader may notice that the above proposition states that the concept of Minkowski normal curvature is somehow analogous to it. Moreover, the relations between the affine normal curvature and the affine shape operator are very similar to the relations between the Minkowski normal curvature and the derivative of the Birkhoff-Gauss map, see \cite[Section 5]{davis} and \cite[Section 5]{diffgeom}.\\
\end{remark}
Let $f:M\rightarrow(\mathbb{R}^3,||\cdot||)$ be a \emph{nondegenerate} immersed surface (meaning that the affine fundamental form $h$ of $M$ has rank $2$), and fix a point $a \in \mathbb{R}^3\setminus M$. The \emph{affine distance function} from $a$ to $M$ is the function $\rho:M\rightarrow\mathbb{R}$ defined by the decomposition
\begin{align}\label{affinedist} p - a = \rho(p)\eta(p) + V(p),
\end{align} 
where $V(p) \in T_pM$. We aim to find a result similar to \cite[Proposition 6.2]{nomizu} for the affine distance function, since this would give another expression for the Minkowski mean curvature. However, since in general our immersion is not a \emph{Blaschke immersion} (see Section \ref{blaschke} for the definition), we may not expect its \emph{cubic form} to vanish, and hence we possibly need a Laplacian concept other than that used in the mentioned result. This is indeed true. We define the \emph{$\nabla$-Laplacian} of a function $f:M\rightarrow\mathbb{R}$ to be
\begin{align*} \Delta f := \mathrm{div}_{\nabla}(\mathrm{grad}_hf),
\end{align*}
where $\mathrm{grad}_hf:M\rightarrow TM$ is the \emph{gradient of $f$ with respect to $h$}, defined to be the (unique) section of $TM$ such that $Xf = h(X,\mathrm{grad}_hf)$, and $\mathrm{div}_{\nabla}:C^{\infty}(TM)\rightarrow C^{\infty}(M)$ is the \emph{divergence operator with respect to the induced connection $\nabla$}, defined formally as
\begin{align*} \mathrm{div}_{\nabla}X|_p = \mathrm{tr}\{Y\mapsto \nabla_YX:Y\in T_pM\}
\end{align*}
for sections $X \in C^{\infty}(TM)$. We can re-obtain the Minkowski mean curvature in terms of the $\nabla$-Laplacian of an affine distance function, as we will see next.  

\begin{teo} Let $\rho:M\rightarrow\mathbb{R}$ be the affine distance function with respect to a given point $a \in \mathbb{R}^3 \setminus M$, and assume that $M$ is nondegenerate. Then the following equality holds:
\begin{align*} \Delta\rho = 2(H\rho - 1).
\end{align*}
\end{teo}
\begin{proof} Derivating (\ref{affinedist}) with respect to $X \in T_pM$, we get
\begin{align*} X = (X\rho)\eta + \rho D_X\eta + D_XV = (X\rho)\eta + \rho D_X\eta + \nabla_XV + h(X,V)\eta.
\end{align*}
Since $D_X\eta$ is tangential, we have that $(X\rho) = -h(X,V)$, and from this we get $\mathrm{grad}_h\rho = -V$. We also have 
\begin{align}\label{eqaff} X = \rho D_X\eta + \nabla_XV.
\end{align}
To calculate the divergence, we may use any positive definite bilinear form, and hence we use $b$. As usual, let $V_1$ and $V_2$ be principal directions of $M$ at $p$ associated to principal curvatures $\lambda_1,\lambda_2 \in \mathbb{R}$, respectively. Assume that both vectors are normalized with respect to $b$. We have
\begin{align*} \mathrm{div}_{\nabla}(\mathrm{grad}_h\rho) = b(-\nabla_{V_1}V,V_1) + b(-\nabla_{V_2}V,V_2)  = b(\rho\lambda_1V_1 - V_1,V_1) + b(\rho\lambda_2V_2 - V_2, V_2) = \\ = \rho\lambda_1 + \rho\lambda_2 - 2 = 2(H\rho - 1),
\end{align*}
where the second equality comes from (\ref{eqaff}). 

\end{proof}
\begin{remark} We say that an immersed surface is \emph{minimal} (in the Minkowski sense) if its Minkowski mean curvature vanishes everywhere. The main interest in the above theorem is that it means, in particular, that a nondegenerate immersed surface $M$ is minimal if and only if $\Delta\rho = -2$. This characterizes nondegenerate Minkowski minimal surfaces in terms of a partial differential equation for the affine distance function. \\
\end{remark}

In affine differential geometry, it is known that the affine distance function from a fixed point $a \in \mathbb{R}^{n+1}$ to a nondegenerate Blaschke hypersurface $M$ is constant if and only if $M$ is a \emph{proper affine hypersphere} with center $a$ (see \cite[Proposition 5.10]{nomizu}). Next we prove something analogous for normed spaces, characterizing Minkowski spheres.

\begin{prop} Let $f:M\rightarrow(\mathbb{R}^3,||\cdot||)$ be a nondegenerate surface immersion, and let $a \in \mathbb{R}^3$. Denote by $\rho:M\rightarrow\mathbb{R}$ the affine distance function from $a$ to $M$, as defined in \emph{(\ref{affinedist})}. Then, $\rho$ is constant if and only if $M$ is contained in a Minkowski sphere with $a$ as center. 
\end{prop}  
\begin{proof} The ``only if" part is immediate. Differentiating (\ref{affinedist}) in $p \in M$ and with respect to a direction $X \in T_pM$ yields the equality
\begin{align*} X = X(\rho)\eta + \rho D_X\eta + \nabla_XV + h(X,V)\eta,
\end{align*}
and hence $X(\rho) = -h(X,V)$. Since $M$ is nondegenerate, if $\rho$ is constant we have $V = 0$. Therefore, $X = \rho D_X\eta$, and hence
\begin{align*} d\eta_p(X) = \frac{1}{\rho}X,
\end{align*} 
for any $p \in M$ and $X \in T_pM$. As a consequence, all points of $M$ are umbilic points, and this characterizes Minkowski spheres (see \cite[Proposition 4.5]{diffgeom}).

\end{proof}

\section{The Birkhoff normal as the affine normal} \label{blaschke}

A transversal vector field $\xi$ in an immersed surface $f:M\rightarrow(\mathbb{R}^3,||\cdot||)$ induces two natural volume elements. The \emph{induced volume} is the $2$-form given by $\omega(X,Y):=\mathrm{det}[X,Y,\xi]$, where $\mathrm{det}$ denotes the usual determinant in $\mathbb{R}^3$. Also, the affine fundamental form $h$ associated to $\xi$ induces the volume form $\omega_h(X,Y):=|\mathrm{det}[h_{ij}]|^{1/2}$, where $[h_{ij}]$ is the matrix of $h$ with respect to the vectors $X$ and $Y$. A \emph{Blaschke immersion} is an immersion endowed with an equiaffine transversal vector field for which $|\omega| = \omega_h$, where $|\cdot|$ denotes the usual absolute value in $\mathbb{R}$ (see \cite{nomizu} for more information on Blaschke immersions).\\

It is well known that for a nondegenerate immersed surface $f:M\rightarrow \mathbb{R}^3$ there exists (locally) a transversal vector field that makes $f$ a Blaschke immersion, and that this vector field is unique up to the sign (see \cite{nomizu} for a proof). We call it the \emph{affine normal field} of $M$. It is also clear that the affine normal field can be globally defined if and only if $M$ is orientable.  The natural question that arises here is: when is the Birkhoff normal vector field the affine normal field of an immersed surface? A first result in this direction discusses whether the Birkhoff normal of the unit sphere is the affine normal. This question is independent of Minkowski geometry: we are asking whether the position vector of a centrally symmetric smooth and strictly convex body is its affine normal. We consider first the planar case, where the immersed hypersurfaces are curves. 

\begin{teo} The unit circle of a normed plane with the Birkhoff normal field as transversal field is a Blaschke immersion if and only if the plane is Euclidean. 
\end{teo}
\begin{proof} Assume that we have a usual auxiliary Euclidean structure in $\mathbb{R}^2$. Let $\varphi(s):[0,l_e(S)]\rightarrow (\mathbb{R}^2,||\cdot||)$ be a parametrization of the unit circle by \textbf{Euclidean} arc-length, where $l_e(S)$ denotes the Euclidean length of the unit circle. The Gauss equation reads
\begin{align*} \varphi''(s) = f(s)\varphi'(s) + h(s)\varphi(s),
\end{align*}
for some functions $f,h:[0,l_e(S)]\rightarrow\mathbb{R}$ (here, $f\varphi'$ is the induced connection). Let $\xi:[0,l_e(S)]\rightarrow \mathbb{R}^2$ be the Euclidean unit normal field (outward pointing). Taking inner products, we have
\begin{align*} h = \frac{\langle\varphi'',\xi\rangle}{\langle \varphi,\xi\rangle},
\end{align*}
where we omit the parameter for simplicity. Since $\xi$ is unit and $\varphi$ is a parametrization on Euclidean arc-length, it follows that $\langle \varphi'',\xi\rangle = -k_e$, the Euclidean curvature of $S$. Notice that, moreover, the function $g:=\langle\varphi,\xi\rangle$ is the usual support function of $S$. Now, if $\varphi$ is the affine normal, we get the equality
\begin{align*} \frac{k_e}{g} = |h| = [\varphi',\varphi]^2 = [\varphi',\langle\varphi,\xi\rangle\xi + \langle\varphi',\varphi\rangle\varphi']^2 = \langle\varphi,\xi\rangle^2 = g^2.
\end{align*}
It follows that $k_e = g^3$. Let now $\theta$ be the parameter of $S$ by the angle between the tangent direction and the $x$-axis. Then it is known that the Euclidean curvature is given in terms of the support function by 
\begin{align*}k_e^{-1} = \frac{d^2g}{d\theta^2} + g,
\end{align*}
and hence the Euclidean support function of the unit circle is a $\pi$-periodic solution of the Ermakov-Pinney equation (see \cite{pinney}), namely
\begin{align*} \frac{d^2g}{d\theta^2} + g = g^{-3}.
\end{align*}
Therefore, by uniqueness it follows that $g = 1$ (cf. \cite[Theorem 5.1]{Ba-Ma-Sho}). Recalling that $g$ is the Euclidean support function of $S$, we have that $S$ is indeed the Euclidean unit circle. 

\end{proof}

Now we prove the three-dimensional version of the previous theorem. 

\begin{teo} Let $||\cdot||$ be an admissible norm in $\mathbb{R}^3$. Then the Birkhoff normal vector field of the unit sphere $\partial B$ is the affine normal field if and only if the norm is derived from an inner product.
\end{teo}
\begin{proof} Before we start, we underline that our proof is quite independent of the theory developed here, which is used only as inspiration. Let $p \in \partial B$, and $E_1,E_2\in C^{\infty}(U)$ be vector fields given by the (Euclidean) orthonormal principal directions of $\partial B$ at each point. Also, let $\lambda_1,\lambda_2$ denote the (Euclidean) principal curvatures of $\partial B$ at each point. If $\eta$ denotes the Birkhoff-Gauss map of $\partial B$ (which is precisely the position vector, since we are considering the geometry given by $\partial B$), then equality (\ref{exph}) gives the following equations:
\begin{align*} h(E_1,E_1) = -\frac{\langle E_1,d\xi_qE_1\rangle}{\langle\eta,\xi\rangle} = -\frac{\lambda_1}{\langle\eta,\xi\rangle} \ \ \mathrm{and}\\
h(E_2,E_2) = -\frac{\langle E_2,d\xi_qE_2\rangle}{\langle\eta,\xi\rangle} = -\frac{\lambda_2}{\langle\eta,\xi\rangle},
\end{align*}
for any $q \in U$, where $h$ is, as usual, the affine fundamental form induced by the transversal vector field $\eta$. Also, since $E_1$ and $E_2$ are conjugate directions, it follows that $D_{E_1}E_2$ is tangential. Therefore, we have $h(E_1,E_2) = 0$. Thus, the volume form induced by $h$ is given by
\begin{align*} \omega_h(E_1,E_2) = \left(\frac{\lambda_1\lambda_2}{\langle\eta,\xi\rangle^2}\right)^{\frac{1}{2}} = \frac{(\lambda_1\lambda_2)^{\frac{1}{2}}}{\langle\eta,\xi\rangle}.
\end{align*}
On the other hand, the volume form $\omega$ induced by the Birkhoff normal vector field is defined as
\begin{align*} \omega(E_1,E_2) = \mathrm{det}[E_1,E_2,\eta] = \langle\eta,\xi\rangle,
\end{align*}
where, up to a re-orientation, we may assume $\langle\eta,\xi\rangle > 0$. By definiton, $\eta$ is the affine normal if and only if $\omega_h = \omega$. Then the assumption that $\eta$ is the affine normal of $\partial B$ yields
\begin{align*} \langle\eta,\xi\rangle = (\lambda_1\lambda_2)^{\frac{1}{4}}.
\end{align*}
Notice that $\lambda_1\lambda_2$ is the (Euclidean) Gaussian curvature of $\partial B$. Let us use the notation $K := \lambda_1\lambda_2$. Following \cite[Example 3.4]{nomizu}, the affine normal of $\partial B$ must be given by
\begin{align}\label{affnoreq} \eta = K^{\frac{1}{4}}\xi + Z,
\end{align}
where $\xi$ is the Euclidean normal vector, and $Z$ is the gradient of the function $K^{\frac{1}{4}}$ with respect to $h$. Thus, $Z\in T_qM$ is, at each point $q \in U$, the vector for which
\begin{align*} h(Z,X) = X\left(K^{\frac{1}{4}}\right)
\end{align*}
holds for every smooth vector field $X$. Since $K^{\frac{1}{4}} = \langle\eta,\xi\rangle$, this equality reads
\begin{align*} h(Z,X) = X(\langle\eta,\xi\rangle) = \langle D_X\eta,\xi\rangle + \langle \eta,D_X\xi\rangle  = \langle \eta,d\xi_qX\rangle.
\end{align*}
Then, replacing $X$ by $E_1$ and $E_2$, respectively, and using again formula (\ref{exph}), we have
\begin{align*} \langle Z,E_1\rangle = -\lambda_1\langle\eta,\xi\rangle\langle \eta,E_1\rangle \ \ \mathrm{and} \\
\langle Z,E_2\rangle = -\lambda_2\langle\eta,\xi\rangle\langle\eta,E_2\rangle.
\end{align*}
Finally, (\ref{affnoreq}) and the decomposition of $Z$ in the basis $\{E_1,E_2\}$ and of $\eta$ in the basis $\{E_1,E_2,\xi\}$ yield
\begin{align*} K^{\frac{1}{4}}\xi + Z = \langle\eta,\xi\rangle\xi -\lambda_1\langle\eta,\xi\rangle\langle \eta,E_1\rangle E_1 -\lambda_2\langle\eta,\xi\rangle\langle\eta,E_2\rangle E_2 = \\ = \langle\eta,\xi\rangle\xi + \langle\eta,E_1\rangle E_1 + \langle\eta,E_2\rangle E_2.
\end{align*}
Therefore, we have that
\begin{align*} \lambda_1 = \lambda_2 = -\frac{1}{\langle\eta,\xi\rangle}.
\end{align*}
It follows that every point of $\partial B$ is umbilic (in the Euclidean sense). Therefore, $\partial B$ must be a Euclidean circle.

\end{proof}

Our next concern is an existence problem: Let $||\cdot||$ be an admissible norm yielding a Minkowski geometry in $\mathbb{R}^3$. Can we always guarantee that an immersed surface exists which, when endowed with the Birkhoff-Gauss map, is a Blaschke immersion? We show now that the answer is no, except for the Euclidean subcase.

\begin{teo} Assume that $f:M\rightarrow(\mathbb{R}^3,||\cdot||)$ is a connected, compact and immersed surface without boundary, where $||\cdot||$ is admissible. Then the affine normal of $M$ equals its Birkhoff normal if and only if the norm is Euclidean and $M$ is a sphere.
\end{teo}
\begin{proof} Let $K$ be the (Euclidean) Gaussian curvature of $M$, and let $\eta:M\rightarrow\partial B$ be its Birkhoff normal field. If $\eta$ is the affine normal field, then from the proof of the last theorem we have
\begin{align*} \eta = K^{\frac{1}{4}}\xi + Z,
\end{align*}
where $Z \in T_pM$ is the vector for which $h(Z,X) = X(K^{\frac{1}{4}})$ holds for each $X \in T_pM$. Our result will be a consequence of the fact that the derivatives of $\eta$ and $\xi$ are always tangential. Indeed, for any $p \in M$ and $X \in T_pM$ we have the equality
\begin{align*} D_X\eta = X(K^{\frac{1}{4}})\xi + K^{\frac{1}{4}}D_X\xi + D_XZ = h(Z,X)\xi + K^{\frac{1}{4}}D_X\xi  + \nabla_XZ + h(X,Z)\eta,
\end{align*}
 and hence $0 = \langle D_X\eta,\xi\rangle = h(Z,X)(1+\langle\eta,\xi\rangle)$. Since $\langle\eta,\xi\rangle$ can be assumed to be positive (after re-orienting $\eta$, if necessary), it follows that $h(Z,X) = 0$ for any $p \in M$ and $X \in T_pM$. Therefore, the Euclidean Gaussian curvature of $M$ is constant. It follows that $M$ must be contained in a Euclidean sphere. Also we have that $\langle\eta,\xi\rangle$ is constant, since $\langle\eta,\xi\rangle = K^{\frac{1}{4}} = c$, say. Since $\eta$ can be regarded as the position vector of $\partial B$, it follows that the (Euclidean) support function of $\partial B$ is constant. Thus, $M$ is a Euclidean sphere.

\end{proof}

\begin{remark} If we drop the hypothesis of $M$ being compact, then we would clearly still have that $M$ is contained in a Euclidean sphere, and that there is a ``portion" of $\partial B$ which is a piece of a Euclidean sphere. More precisely, this would be the subset of $\partial B$ of points at which $\partial B$ is supported by a hyperplane parallel to some tangent space of $M$.
\end{remark}

\bibliography{bibliography.bib}

\end{document}